\documentclass[preprint,12pt]{elsarticle}

\usepackage{amssymb}
\usepackage{amsmath}
\usepackage{amsthm}

\newtheorem{theorem}{Theorem}

\newtheorem{corollary}[theorem]{Corollary}
\newtheorem{lemma}[theorem]{Lemma}

\newtheorem{example}[theorem]{Example}

\journal{Linear Algebra and its Applications}

\begin{document}

\begin{frontmatter}



\title{A Note on the Bias and Kemeny's Constant\\ in Markov Reward Processes with an Application to Markov Chain Perturbation}


\author{Ronald Ortner}

 \affiliation{organization={Montanuniversit{ae}t Leoben},
             addressline={Franz-Joseph-Strasse 18},
             city={Leoben},
             postcode={8700},
             country={Austria}}

\begin{abstract}
Given a unichain Markov reward process (MRP), we provide an explicit expression for the bias values in terms of mean first passage times. This result implies a generalization of known Markov chain perturbation bounds for the stationary distribution to the case where the perturbed chain is not irreducible. It further yields an improved perturbation bound in 1-norm. As a special case, Kemeny's constant can be interpreted as the translated bias in an MRP with constant reward 1, which offers an intuitive explanation why it is a constant.
\end{abstract}



\begin{keyword}
Markov reward process \sep
Markov chain \sep 
bias \sep
perturbation theory \sep
stationary distribution \sep
mean first passage times \sep
Kemeny's constant 


\MSC[2020] 60J10, 90C40

\end{keyword}

\end{frontmatter}



\section{Preliminaries}
 
 \subsection{Markov reward processes}
 
Consider a Markov chain over a finite state space $S$ with states $1,2,\ldots,N$ and transition probabilities $p_{ij}$ ($1\leq i,j\leq N$). We assume in what follows that the Markov chain is unichain, that is, it consists of a single recurrent class and a possibly empty set of transient states. Equipping the Markov chain with a reward function $r:S\to \mathbb{R}$ yields a Markov reward process (MRP), cf.\ Section 8.2 of \cite{puterman} for the following facts. Usually, it is assumed that the reward\footnote{For functions $f:S\to \mathbb{R}$ we write in the following short $f_i$ instead of $f(i)$.} $r_i$ in each state $i$ is the mean of some fixed reward distribution. The \textit{average reward} $\rho$ in the MRP is then defined as 
\[
    \rho = \lim_{T\to\infty} \frac{1}{T} \sum_{t=1}^T \mathbb{E}[r_{S_t}|S_1=i],
\]
where $S_t$ is a random variable for the state at step $t$. In the assumed case of a unichain MRP 
the value $\rho$ is independent of the initial state $i$. In fact, $\rho$ can be written in terms of the stationary distribution~$\mu$ as
\[
    \rho = \sum_{i=1}^N \mu_i\, r_i,
\]
noting that a unchain Markov chain has a unqiue stationary distribution~$\mu$ with $\mu_i=0$ for transient states $i$. 

 \subsection{The bias}
While $\rho$ is the average reward in the limit, the actual collected rewards will differ depending on the initial state. This is made precise by the notion of \textit{bias}, which for each state $i$ is defined as 
\begin{equation}\label{eq:bias}
 \lambda_i = \mathbb{E}\bigg[\sum_{t=1}^\infty \big(r_{S_t}-\rho\big)\,\Big|\, S_1=i \bigg]
\end{equation}
in MRPs with underlying aperiodic chain, while in general one sets
\begin{equation}\label{eq:bias-per}
 \lambda_i = \lim_{T\to \infty}\frac{1}{T}\sum_{\tau=1}^T \mathbb{E}\bigg[\sum_{t=1}^\tau \big(r_{S_t}-\rho\big)\,\Big|\, S_1=i \bigg].
\end{equation}
By definition, the difference of two bias values $\lambda_i-\lambda_j$ quantifies the advantage in accumulated reward when starting in state $i$ over starting in state $j$.

\begin{example}\label{ex}
Assume that all states $i$ have the same mean reward $r_i=r$. 
Then the average reward $\rho=r$ and $\rho$ is attained from the first step, independent of the initial state. Accordingly, the difference of any two bias values $\lambda_i-\lambda_j$ has to be 0. 
 Indeed, the bias is 0 for all states.
\end{example}

The bias values of an MRP are a solution of the \textit{Poisson equation}, that is, for all $i$,
\begin{equation}\label{eq:pe}
 \rho + \lambda_i = r_i + \sum_{j=1}^N p_{ij}\, \lambda_j.
\end{equation}
The bias values $\lambda_i$ in addition satisfy $\sum_i \mu_i \lambda_i = 0$, which can be achieved for any solution of \eqref{eq:pe} by adding a suitable vector with identical entries. On the other hand, any respective translation $\lambda'_i:=\lambda_i + c$ of the bias values $\lambda_i$ still fulfills the Poisson equation. 

 \subsection{Mean first passage times}
The \textit{mean passage time} $\tau_{ij}$ ($i\neq j$) is defined as the expected time it takes to first visit $j$ when starting in $i$. Further, the \textit{mean return time} $\tau_{ii}$ is the first time $i$ is revisited again when starting in~$i$. It is well-known~\cite{keme} that in irreducible Markov chains, for $1\leq i,j\leq N$
\begin{eqnarray}
 \tau_{ij} &=& 1 + \sum_{k\neq j} p_{ik}\, \tau_{kj}, \mbox{ and }\label{eq:mpt}  \\
 \tau_{ii} &=& \frac{1}{\mu_i}\label{eq:mrt}.
\end{eqnarray}

In unichain Markov chains $\tau_{ij}$ can be infinite for transient states $j$. However, all $\tau_{ij}$ are finite for recurrent states $j$ and for these \eqref{eq:mpt} and \eqref{eq:mrt} still hold.

\section{Main result}

Our main result gives an explicit expression for the bias values in terms of the mean first passage times of an MRP. 

\begin{theorem}\label{thm}
 The values 
 \[
     \lambda'_i := -\sum_{j\neq i} \mu_j\, r_j\, \tau_{ij}
 \]
 satisfy the Poisson equation \eqref{eq:pe}. 
\end{theorem}

\begin{proof}
Inserting the defined values $\lambda'_i$ in the right hand side of the the Poisson equation \eqref{eq:pe}, we obtain, using \eqref{eq:mpt} and \eqref{eq:mrt},
\begin{eqnarray*} 
 r_i + \sum_j p_{ij}\, \lambda'_j  &=& r_i - \sum_j p_{ij} \sum_{k\neq j} \mu_k\, r_k\, \tau_{jk}  \\
   &=& r_i - \sum_j p_{ij} \sum_{k} \mu_k\, r_k\, \tau_{jk} + \sum_j p_{ij}\, \mu_j\, r_j\, \tau_{jj}\\
   &=& r_i - \sum_{k} \mu_k\, r_k\, \sum_j p_{ij}\, \tau_{jk} + \sum_j p_{ij}\, r_j \\
   &=& r_i - \sum_{k} \mu_k\, r_k \big( \tau_{ik} -1 + p_{ik}\,\tau_{kk}  \big) + \sum_j p_{ij}\, r_j 
\end{eqnarray*}
\begin{eqnarray*} 
   &=& r_i - \sum_{k} \mu_k\, r_k \tau_{ik} + \sum_{k} \mu_k\, r_k - \sum_{k} \mu_k\, r_k\, p_{ik}\,\tau_{kk}   + \sum_j p_{ij}\, r_j \\
&=& r_i + \lambda'_i - \mu_i\, r_i\, \tau_{ii} + \sum_{k} \mu_k\, r_k - \sum_{k} r_k\, p_{ik} + \sum_j p_{ij}\, r_j \\
    &=& \lambda'_i + \rho, 
\end{eqnarray*}

which concludes the proof. 
\end{proof}

In order to obtain the actual bias values $\lambda_i$ from the $\lambda'_i$ defined in Theorem~\ref{thm}, these have to be translated, cf.\ the remark after \eqref{eq:pe}.

\section{Implications}

While Theorem \ref{thm} is quite simple, it has some interesting implications discussed in the following.

\subsection{Bias span}
A known connection between the bias and transition times is the following. We define the \textit{diameter} $D:= \max_{i\neq j}\tau_{ij}$ to be the maximal mean first passage time between two states. 
Then for rewards bounded in $[0,1]$ the \textit{bias span} ${\rm span}(\lambda)$ is upper bounded as
\begin{equation}\label{eq:bsd}
 {\rm span}(\lambda) \,:= \max_i \lambda_i - \min_i \lambda_i \,\leq\, D.
\end{equation}
This observation has been made in the more general context of Markov decision processes (MDPs), see \cite{jaorau}. Theorem \ref{thm} makes the connection between bias and transition times precise. Note that \eqref{eq:bsd} is a straightforward consequence of Theorem \ref{thm}.

\subsection{Markov chain perturbation}
Let us consider a Markov chain with transition matrix $P=(p_{ij})_{i,j=1}^N$ and a perturbed chain with transition matrix $\tilde{P}=(\tilde{p}_{ij})_{i,j=1}^N$.
Perturbation bounds for the stationary distribution provide inequalities of the form
\[
     \|\mu-\tilde{\mu}\|_p  \,\leq\,  \kappa \cdot \|P-\tilde{P}\|_q   
\]
for so-called \textit{condition numbers} $\kappa$ (i.e., parameters of the unperturbed chain),
most commonly for $p=1,\infty$ and $q=\infty$, cf.\ \cite{chome2} for an overview. The condition numbers of the following two bounds involve mean first passage times and are closely related to the bias values of Theorem \ref{thm}.

\begin{theorem}[Cho\,\&\,Meyer\,\cite{chome}]\label{thm:chome}
Let $P$, $\tilde{P}$ be the transition matrices of two irreducible Markov chains with stationary distributions $\mu$ and $\tilde{\mu}$. Then
\[
   |\mu_i-\tilde{\mu}_i|  \,\leq\,  \frac{\mu_i}{2} \cdot \max_{j\neq i} \tau_{ji} \cdot \|P-\tilde{P}\|_\infty .  
\]
\end{theorem}
The condition number of the following bound uses Kemeny's constant $\eta$, defined as
\[
    \eta \,:=\, \eta_i \,:=\, \sum_{j\neq i} \mu_j\,\tau_{ij}.
\]
It can be shown that $\eta_i$ is indeed independent of $i$ (cf.\ also next section below).
Note that $\eta_i$ coincides with $\lambda'_i$ when all rewards are 1.

\begin{theorem}[Hunter\,\cite{hunter2}]\label{thm:hunter}
Let $P$, $\tilde{P}$ be the transition matrices of two irreducible Markov chains with stationary distributions $\mu$ and $\tilde{\mu}$. Then
\[
   \|\mu-\tilde{\mu}\|_1  \,\leq\,  \frac{\eta}{2} \cdot \|P-\tilde{P}\|_\infty .  
\]
\end{theorem}

The bounds of Theorems \ref{thm:chome} and \ref{thm:hunter} have been shown for irreducible Markov chains. In the more general setting of MDPs, perturbation bounds are known that hold more generally in structures that need not be irreducible \cite{oams}. The respective condition number is the diameter, which is larger than the condition numbers used in Theorems \ref{thm:chome} and \ref{thm:hunter}. However, the diameter only serves as an upper bound on the bias span as in \eqref{eq:bsd}. Accordingly, with the result of Theorem \ref{thm}, we can obtain perturbation bounds which are not only more general but also sharper. 

Let us first restate the perturbation bound of \cite{oams} for the case of MRPs, a proof is given in the appendix.\footnote{The proof of the original bound is contained in an unpublished appendix of \cite{ro-hbs}. This bound is stated in a very general context when the state spaces of the original and the perturbed MDP need not coincide and also the reward function may be perturbed. For the case of two MDPs with the same state space, the proof has been restated in \cite{ro-hbs}.}

\begin{theorem}[Ortner et al.\,\cite{oams}]\label{thm:bsp}
 Consider a unichain MRP with transition matrix $P$ and another MRP with the same reward function $r$ but a (possibly not irreducible) perturbed matrix $\tilde{P}$. 
 Then, independent of the initial state, the difference of the average rewards $\rho$, $\tilde{\rho}$ of the two MRPs is upper bounded as
\[
   |\rho - \tilde{\rho}|  \,\leq\, \tfrac{1}{2}\cdot{\rm span}(\lambda) \cdot \|P-\tilde{P}\|_\infty   .
\]
\end{theorem}

Theorem \ref{thm:bsp} easily implies Theorems \ref{thm:chome} and \ref{thm:hunter}, but now these  hold more generally for the case when the original Markov chain is unichain, and there are no conditions on the perturbed chain. 

\begin{proof}[Proof of Theorem \ref{thm:chome} from Theorem \ref{thm:bsp}]
We fix an initial state and note that $\tilde{\mu}$ and $\tilde{\rho}$ depend on this initial state in the following. Set the reward function in Theorem \ref{thm:bsp} to be $r_i=1$ and $r_j=0$ for all $j\neq i$. 
Then by definition of $\lambda'_i$,
\begin{equation}\label{eq:q}
 \lambda'_j = \left\{  \begin{array}{cl} 0 & \mbox{for } j=i , \\ 
                                        - \mu_i \tau_{ji} & \mbox{for } j\neq i ,
                \end{array} \right.
\end{equation}
so that 
\[
   {\rm span}(\lambda') =  \mu_i \max_{j\neq i} \tau_{ji}.
\]
By Theorems \ref{thm:bsp} and \ref{thm},
\begin{eqnarray*}\label{eq:p}
     |\mu_i - \tilde{\mu}_i|  \,=\,  |\rho - \tilde{\rho}\,|  &\leq& \tfrac{1}{2}\,{\rm span}(\lambda) \cdot \|P-\tilde{P}\|_\infty ,\\
     &=& \tfrac{1}{2}\,{\rm span}(\lambda') \cdot \|P-\tilde{P}\|_\infty ,\\
     &=& \tfrac{1}{2}\, \mu_{i} \max_{j\neq i} \tau_{ji} \cdot \|P-\tilde{P}\|_\infty,   
\end{eqnarray*}
which is precisely the bound of Theorem \ref{thm:chome} and holds independent of the chosen initial state.
\end{proof}

\begin{proof}[Proof of Theorem \ref{thm:hunter} from Theorem \ref{thm:bsp}]
Again we fix an initial state on which $\tilde{\mu}$ and $\tilde{\rho}$ depend in the following.
We define a reward function 
\begin{equation}\label{eq:r}
 r_i = \left\{  \begin{array}{cl} 1 & \mbox{if } \mu_i\geq \tilde{\mu}_i, \\ 
                                         0 & \mbox{otherwise} . 
                \end{array} \right.
\end{equation}
Then the difference of the average rewards $\rho$, $\tilde{\rho}$ of the original and the perturbed MRP is the total variation distance between $\mu$ and $\tilde{\mu}$, which is known to be $\frac{1}{2}\|\mu-\tilde{\mu}\|_1$.
Therefore, we get by Theorems \ref{thm:bsp}, independent of the initial state,
\begin{equation}
  \tfrac{1}{2} \, \|\mu-\tilde{\mu}\|_1 \,=\, |\rho - \tilde{\rho}\,| \,\leq\, \tfrac{1}{2}\, {\rm span}(\lambda) \cdot \|P-\tilde{P}\|_\infty .
\end{equation}
For ${\rm span}(\lambda)$ we have by Theorem \ref{thm},
\begin{eqnarray}
     {\rm span}(\lambda) =  {\rm span}(\lambda') 
          &=&  \max_i \sum_{\substack{j\neq i\\ \mu_j\geq \tilde{\mu}_j}} \mu_j \tau_{ij}
             - \min_i \sum_{\substack{j\neq i\\ \mu_j\geq \tilde{\mu}_j}} \mu_j \tau_{ij} \label{eq:bias-span-exact}\\
        &\leq&  \max_i \sum_{\substack{j\neq i\\ \mu_j\geq \tilde{\mu}_j}} \mu_j \tau_{ij}
        \,\leq\,  \max_i \sum_{j\neq i} \mu_j \tau_{ij}  \,=\eta, \nonumber
\end{eqnarray}
which finishes the proof.
\end{proof}

Looking at the proofs, we see that while we precisely obtain the bound of Theorem \ref{thm:chome}, the bound of Theorem \ref{thm:hunter} is a bit loose when compared to the bound implied by Theorems \ref{thm:bsp} and \ref{thm}. The following corollary to Theorem~\ref{thm:bsp} summarizes our findings and presents a respective improved bound on $\|\mu-\tilde{\mu}\|_1$.

\begin{corollary}
 Consider a unichain Markov chain with transition matrix $P$ and stationary distribution $\mu$, and a perturbed Markov chain with transition matrix $\tilde{P}$, which may be not irreducible. Then independent of the initial state, the stationary distribution $\mu'$ of the perturbed chain satisfies
 \begin{eqnarray*}
  |\mu_i-\tilde{\mu}_i|  &\leq&  \frac{\mu_i}{2} \cdot \max_{j\neq i} \tau_{ji} \cdot \|P-\tilde{P}\|_\infty, \mbox{ and} \\
  \|\mu-\tilde{\mu}\|_1  &\leq& \tfrac{1}{2}\cdot \!\!\max_{A\subseteq \{1,2,\ldots, N\}} 
         \Big\{ \max_i\!\! \sum_{j \in A\setminus \{i\}}\!\! \mu_j \tau_{ij}
             - \min_i \!\! \sum_{j \in A\setminus \{i\}}\!\! \mu_j \tau_{ij} \Big\} \cdot \|P-\tilde{P}\|_\infty .
 \end{eqnarray*}
\end{corollary}
\begin{proof}
 The first statement is just Theorem \ref{thm:chome} generalized, which we have shown before. The second statement follows from the proof of Theorem \ref{thm:hunter} above, considering the maximal possible expression on the right hand side of~\eqref{eq:bias-span-exact}.
\end{proof}

\subsection{Kemeny's constant}
We conclude with a remark on Kemeny's constant.
When all rewards $r_i$ are 1, then $\lambda'_i=\eta_i$ for all $i$. As discussed in Example~\ref{ex}, identical rewards imply identical bias values so that it follows that all the $\eta_i$ have to be identical. This not only provides a short proof that $\eta_i=\eta$ for all $i$, it also gives a simple explanation why Kemeny's constant is a constant: The $\eta_i$ are the translated bias values in an MRP with identical rewards and hence have to be identical, too.

\appendix

\section{Proof of Theorem \ref{thm:bsp}}
\label{app1}

We start with a result that after taking $\ell$~steps in the perturbed MRP 
compares the accumulated rewards to the quantity~$\ell \rho$.

\begin{lemma}\label{lem}
Consider a unichain MRP with transition matrix $P$, stationary distribution $\mu$, and bias $\lambda$, and let another MRP have the same reward function $r$ but a perturbed transition matrix $\tilde{P}$. We take $\ell$ steps in the perturbed MRP and write $v_i$ for the number of visits in state $i$. Then it holds with probability at least $1-\delta$ that 
 \begin{equation}\label{eq:aggubo}
     \ell \rho - \sum_{i} v_i \cdot r_i  \;\leq\;  
            \tfrac{\ell}{2}\, {\rm span}(\lambda)\cdot \|P-\tilde{P}\|_\infty + {\rm span}(\lambda)\,\Big(1+ \sqrt{2\ell\log(1/\delta)}\Big),
 \end{equation}
 independent of the initial state.
\end{lemma}

\begin{proof}
We first apply a translation $\bar{\lambda}_i:=\lambda_i - \frac{1}{2}(\max_j \lambda_j + \min_j \lambda_j)$ to the bias values $\lambda_i$. Then 
\begin{equation}\label{eq:bb}
   \|\bar{\lambda}\|_\infty \,=\max_j \bar{\lambda}_j \,=\, \tfrac{1}{2}\, {\rm span}(\lambda)
                  \,=\, \tfrac{1}{2}\,{\rm span}(\bar{\lambda}) .
\end{equation}
Further, the $\bar{\lambda}_i$ still satisfy the Poisson equation \eqref{eq:pe}, so that 
\begin{eqnarray} 
\lefteqn{\ell\rho -  \sum_{i} v_i \cdot r_i  =  \sum_{i} v_i\, \big(\rho - r_i \big) 
=   \sum_{i} v_i \, \Big( \sum_{j}  p_{ij}\,\bar{\lambda}_j - \bar{\lambda}_i  \Big)}  \nonumber  \\
&=&    \sum_{i} v_i \, \Big( \sum_{j}  \tilde{p}_{ij}\,\bar{\lambda}_j - \bar{\lambda}_i  \Big)  
                 +  \sum_{i} v_i \cdot \sum_{j} \big(p_{ij} - \tilde{p}_{ij}\big)\,\bar{\lambda}_j . \label{eq:x}
\end{eqnarray}
Writing $S_t$ for the state at step $t$ we obtain for the first term in \eqref{eq:x}
\begin{eqnarray}
\lefteqn{  \sum_{i} v_i \, \Big( \sum_{j}  \tilde{p}_{ij}\,\bar{\lambda}_j - \bar{\lambda}_i  \Big) 
  =  \sum_{t=1}^{\ell}  \Big( \sum_{j} \tilde{p}_{S_t,j}\, \bar{\lambda}_j - \bar{\lambda}_{S_t}\Big) } \nonumber\\
  &=&  \sum_{t=1}^{\ell}  \Big( \sum_{j} \tilde{p}_{S_t,j}\, \bar{\lambda}_j - \bar{\lambda}_{S_{t+1}}\Big)     
      + \bar{\lambda}_{S_{\ell+1}}  -  \bar{\lambda}_{S_{1}}. \qquad \label{eq:ll4}
\end{eqnarray}
The sequence 
\[
      X_t := \sum_{j} \tilde{p}_{S_t,j}\, \bar{\lambda}_j - \bar{\lambda}_{S_t+1}
\]
is a martingale difference sequence with $|X_t| \leq {\rm span}(\lambda)$, so that
by Azuma-Hoeffding's inequality (e.g., Lemma 10 of \cite{jaorau}) with probability $1-\delta$,
\begin{eqnarray}\label{eq:martingale-ll}
  \sum_{t=1}^{\ell}  \Big( \sum_{j} \tilde{p}_{S_t,j}\, \bar{\lambda}_j - \bar{\lambda}_{S_{t+1}}\Big)    &\leq&  {\rm span}(\lambda) \sqrt{2\ell\log(1/\delta)}.
\end{eqnarray}
Hence we obtain from \eqref{eq:ll4}
\begin{eqnarray}
 \sum_{i} v_i \, \Big( \sum_{j} \tilde{p}_{ij}\,\bar{\lambda}_j - \bar{\lambda}_i  \Big) 
     \,\leq\,   {\rm span}(\lambda) \sqrt{2\ell\log(1/\delta)} + {\rm span}(\lambda).\label{eq:xx}
\end{eqnarray}
The second term in \eqref{eq:x} can be bounded by \eqref{eq:bb} as
\begin{eqnarray}
 \sum_{i} v_i \cdot \sum_{j} \big(p_{ij} - \tilde{p}_{ij}\big)\,\bar{\lambda}_j 
   &\leq& \sum_{i} v_i \cdot \sum_{j} | p_{ij} - \tilde{p}_{ij}|\cdot\big\|\bar{\lambda}\big\|_\infty \nonumber \\
   &\leq& \ell \cdot \|P-\tilde{P}\|_\infty \cdot \tfrac{1}{2}\, {\rm span}(\lambda).  \label{eq:a7}
\end{eqnarray}
Combining \eqref{eq:x}, \eqref{eq:xx}, and \eqref{eq:a7} gives the claimed 
\begin{eqnarray*}
  \ell \rho - \sum_{i} v_i \cdot r_i  \;\leq\;  
            \tfrac{\ell}{2}\, {\rm span}(\lambda)\cdot \|P-\tilde{P}\|_\infty + {\rm span}(\lambda)\,\Big(1+ \sqrt{2\ell\log(1/\delta)}\Big). \qedhere
\end{eqnarray*}
\end{proof}

Now Theorem \ref{thm:bsp} follows from Lemma~\ref{lem} by dividing \eqref{eq:aggubo} by $\ell$, choosing $\delta=1/\ell$, and letting $\ell\to\infty$. \qed






\bibliographystyle{elsarticle-num} 
\bibliography{RL}

\begin{thebibliography}{1}
\expandafter\ifx\csname url\endcsname\relax
  \def\url#1{\texttt{#1}}\fi
\expandafter\ifx\csname urlprefix\endcsname\relax\def\urlprefix{URL }\fi
\expandafter\ifx\csname href\endcsname\relax
  \def\href#1#2{#2} \def\path#1{#1}\fi

\bibitem{puterman}
M.~L. Puterman, Markov Decision Processes: Discrete Stochastic Dynamic
  Programming, John Wiley \& Sons, Inc., New York, NY, USA, 1994.

\bibitem{keme}
J.~Kemeny, J.~Snell, Finite {M}arkov Chains, Van Nostrand, 1960.

\bibitem{jaorau}
T.~Jaksch, R.~Ortner, P.~Auer, Near-optimal regret bounds for reinforcement
  learning, J.\ Mach.\ Learn.\ Res. 11 (2010) 1563--1600.

\bibitem{chome2}
G.~E. Cho, C.~D. Meyer, Comparison of perturbation bounds for the stationary
  distribution of a {M}arkov chain, Linear Algebra Appl. 335 (2001) 137--150.

\bibitem{chome}
G.~E. Cho, C.~D. Meyer, Markov chain sensitivity measured by mean first passage
  times, Linear Algebra Appl. 316 (2000) 21--28.

\bibitem{hunter2}
J.~J. Hunter, Mixing times with applications to perturbed {M}arkov chains,
  Linear Algebra Appl. 417 (2006) 108--123.

\bibitem{oams}
R.~Ortner, O.~Maillard, D.~Ryabko, Selecting near-optimal approximate state
  representations in reinforcement learning, in: Algorithmic Learning Theory --
  25th International Conference, {ALT} 2014, 2014, pp. 140--154.

\bibitem{ro-hbs}
R.~Ortner, Markov chain estimation, approximation, and aggregation for average
  reward {M}arkov decision processes and reinforcement learning, submitted to
  Handbook of Statistics (2024).

\end{thebibliography}

\end{document}